\documentclass[english]{amsart}

\usepackage{amsmath}

\usepackage{amssymb}

\usepackage{amscd} \usepackage{tabularx}
\usepackage[all,cmtip,line]{xy}

\newcommand{\ra}{\rightarrow}
\newcommand{\lra}{\longrightarrow}

\newcommand{\into}{\hookrightarrow}

\newcommand{\iso}{\stackrel{\sim}{\ra}}

\newlength{\ownl}

\newcommand{\norm}{{\mbox{\bf N}}}
\newcommand{\ndiv}{{\mbox{$\not| $}}}

\newcommand{\Art}{{\operatorname{Art}\,}}

\newcommand{\BC}{{\operatorname{BC}\,}}

\newcommand{\Frob}{{\operatorname{Frob}}}
\newcommand{\Gal}{{\operatorname{Gal}\,}}

\newcommand{\Hom}{{\operatorname{Hom}\,}}

\newcommand{\Iw}{{\operatorname{Iw}}}

\newcommand{\WD}{{\operatorname{WD}}}

\newcommand{\Spp}{{\operatorname{Sp}\,}}

\newcommand{\rec}{{\operatorname{rec}}}
\newcommand{\tr}{{\operatorname{tr}\,}}

\newcommand{\nind}{{\operatorname{n-Ind}\,}}

\newcommand{\st}{{\operatorname{st}}}

\newcommand{\ab}{{\operatorname{ab}}}

\newcommand{\semis}{{\operatorname{ss}}}
\newcommand{\Fsemis}{{\operatorname{F-ss}}}

\newcommand{\rig}{{\operatorname{rig}}}
\newcommand{\reg}{{\operatorname{reg}}}

\newcommand{\A}{{\mathbb{A}}}

\newcommand{\C}{{\mathbb{C}}}

\newcommand{\G}{{\mathbb{G}}}

\newcommand{\Q}{{\mathbb{Q}}}
\newcommand{\R}{{\mathbb{R}}}

\newcommand{\Z}{{\mathbb{Z}}}

\newcommand{\gz}{{\mathfrak{z}}}

\newcommand{\barF}{\overline{{F}}}

\newcommand{\barK}{\overline{{K}}}
\newcommand{\barL}{\overline{{L}}}
\newcommand{\barM}{\overline{{M}}}

\newcommand{\barQQ}{\overline{{\Q}}}

\newcommand{\tR}{\widetilde{{R}}}
\newcommand{\tS}{\widetilde{{S}}}

\newcommand{\tu}{\widetilde{{u}}}

 \newcommand{\tsigma   }{\widetilde{\sigma}}

\newcommand{\Qbar}{{\overline{\Q}}}

\def\RCS$#1: #2 ${\expandafter\def\csname RCS#1\endcsname{#2}}
\RCS$Revision: 53 $
\RCS$Date: 2011-03-23 00:44:16 -0400 (Wed, 23 Mar 2011) $

\newcommand{\To}{\longrightarrow}

\newcommand{\onto}{\twoheadrightarrow}
\newcommand{\isoto}{\stackrel{\sim}{\To}}

\newcommand{\bb}{\mathbb}

\newcommand{\Sp}{\operatorname{Sp}}

\newcommand{\HT}{\operatorname{HT}}

\newcommand{\cA}{\mathcal{A}}
\newcommand{\cB}{\mathcal{B}}

\newcommand{\cH}{\mathcal{H}}

\newcommand{\cJ}{\mathcal{J}}

\newcommand{\cO}{\mathcal{O}}

\newcommand{\cT}{\mathcal{T}}

\newcommand{\cZ}{\mathcal{Z}}

\newcommand{\Qlbar}{\overline{\Q}_{l}}

\usepackage{amsthm}

 \newtheorem{ithm}{Theorem}
\newtheorem{thm}{Theorem}[section]

\newtheorem{cor}[thm]{Corollary}

 \theoremstyle{definition}
 \theoremstyle{definition}
 \theoremstyle{remark}

\numberwithin{equation}{subsection}

\theoremstyle{definition}

\begin{document}
\title{Local-global compatibility for $l=p$, II.}

\author{Thomas Barnet-Lamb}\email{tbl@brandeis.edu}\address{Department of Mathematics, Brandeis University}
\author{Toby Gee} \email{gee@math.northwestern.edu} \address{Department of
  Mathematics, Northwestern University} \author{David Geraghty}
\email{geraghty@math.princeton.edu}\address{Princeton University and
  Institute for Advanced Study} \author{Richard Taylor}
\email{rtaylor@math.harvard.edu}\address{Department of Mathematics,
  Harvard University} \thanks{The second author was partially supported
  by NSF grant DMS-0841491, the third author was partially supported
  by NSF grant DMS-0635607 and the fourth author was partially
  supported by NSF grant DMS-0600716 and by the Oswald Veblen and Simonyi Funds at the IAS}  
\begin{abstract}We prove the compatibility at places dividing
$l$ of the local and global Langlands correspondences for the $l$-adic
Galois representations associated to regular algebraic essentially
(conjugate) self-dual cuspidal automorphic representations of $GL_n$
over an imaginary CM or totally real field. We prove this compatibility up to
semisimplification in all cases, and up to Frobenius semisimplification in the
case of Shin-regular weight.
\end{abstract}
\maketitle
\newpage

\section*{Introduction.}
\label{sec:intro}

Our main result is as follows (see Theorem \ref{main-thm} and Corollary \ref{main-cor}).

\begin{ithm}
  \label{thma}
  Let $F$ be an imaginary CM field or totally real field, let $(\Pi,\chi)$ be a
  regular, algebraic, essentially (conjugate) self-dual automorphic
  representation of $GL_m(\A_F)$ and let $\imath:\Qlbar\iso\C$.    If
  $v|l$ is a place of $F$, then
\[ \imath\WD(r_{l,\imath}(\Pi)|_{G_{F_v}})^{\semis}\cong
\rec(\Pi_v\otimes|\det|^{(1-m)/2})^{\semis}.\]
Furthermore, if $\Pi$ has Shin-regular weight, then 
\[ \imath\WD(r_{l,\imath}(\Pi)|_{G_{F_v}})^{\Fsemis}\cong
\rec(\Pi_v\otimes|\det|^{(1-m)/2}).\]
\end{ithm}

Here $r_{l,\imath}(\Pi)$ denotes the $l$-adic representation assocaited to $\Pi$ (and $\imath$); and $\WD(r)$ denotes the Weil-Deligne representation attached to a de Rham $l$-adic representation $r$ of the absolute Galois group of an $l$-adic field; and $\rec$ denotes the local Langlands correspondence; and $\Fsemis$ denotes Frobenius semi-simplification. 
(See Section \ref{sec:irreducibility} for details on the terminology.)
In fact, we prove a slight refinement of this result which gives some
information about the monodromy operator in the case where $\Pi$ does
not have Shin-regular weight; see Section \ref{sec:irreducibility} for
the details of this.

It is important in some applications to have this compatibility at
places dividing $l$; for example, our original motivation for
considering this problem was to improve the applicability of the main
results of \cite{blggt}; in that paper a variety of automorphy lifting
theorems are proved via making highly ramified base changes, and one
loses control of the level of the automorphic representations under
consideration. This control can be recovered if one knows local-global
compatibility, and this is important in applications to the weight
part of Serre's conjecture (cf. \cite{BLGGU2}, \cite{BLGGUn}).

The proof of Theorem \ref{thma} is surprisingly simple, and relies on
a generalisation of a base change trick that we learned from the
papers \cite{kisindefrings} and \cite{MR2538615} (see the proof of
Theorem 4.3 of \cite{kisindefrings} and Section 2.2 of
\cite{MR2538615}). 
The idea is as follows. Suppose that $\Pi$ has
Shin-regular weight. We wish to determine the Weil-Deligne
representation
$\imath\WD(r_{l,\imath}(\Pi)|_{G_{F_v}})^{\Fsemis}$. The monodromy may
be computed after any finite base change, and in particular we may
make a base change so that $\Pi$ has Iwahori-fixed vectors, which is
the situation covered by \cite{blggtlocalglobalI}; so it suffices to
compute the representation of the Weil group $W_{F_v}$. It is
straightforward to check that in order to do so it is enough to
compute the traces of the elements $\sigma\in W_{F_v}$ of nonzero
valuation (that is, those elements which map to a nonzero power of the
Frobenius element in the Galois group of the residue field). This
trace is then computed as follows: one makes a global base change to a
CM field $E/F$ such that there is a place $w$ of $E$ lying over $v$
such that $\BC_{E/F}(\Pi)_w$ has Iwahori-fixed vectors, and $\sigma$
is an element of $W_{E_w}\le W_{F_v}$. By the compatibility of base
change with the local Langlands correspondence, the trace of $\sigma$
on $\imath\WD(r_{l,\imath}(\Pi)|_{G_{F_v}})^{\Fsemis}$ may then
be computed over $E$, where the result follows from
\cite{blggtlocalglobalI}.

The subtlety in this argument is that the field $E/F$ need not be
Galois, so one cannot immediately appeal to solvable base
change. However, it will have solvable normal closure, so that by a
standard descent argument due to Harris it is enough to know that for
some prime $l'$, the global Galois representation
$r_{l',\imath'}(\Pi)$ is irreducible. Under the additional assumption
that $\Pi$ has extremely regular weight, the existence of such an $l'$
is established in \cite{blggt}. Having thus established Theorem \ref{thma}
in the case that $\Pi$ has extremely regular and Shin-regular weight,
we then pass to the general case by means of an $l$-adic interpolation
argument of Chenevier and Harris, \cite{chenevierharris} and \cite{chenevier}. 
The details are in Section \ref{sec:Chenevier-Harris}.

\subsection*{Notation and terminology.}

We write all matrix transposes on the left; so ${}^t\!A$ is the
transpose of $A$. We let $B_m\subset GL_m$ denote the Borel subgroup
of upper triangular matrices and $T_m\subset GL_m$ the diagonal
torus. We let $I_m$ denote the identity matrix in $GL_m$. 
If $M$ is a field, we let $\barM$ denote an algebraic closure of $M$
and $G_M$ the absolute Galois group $\Gal(\barM/M)$.  Let
$\epsilon_l$ denote the $l$-adic cyclotomic character

Let $p$ be a rational prime and $K/\Q_p$ a finite extension. We let
$\cO_K$ denote the ring of integers of $K$, $\wp_K$ the maximal ideal
of $\cO_K$, $k(\nu_K)$ the residue field $\cO_K/\wp_K$,
$\nu_K:K^\times\onto \Z$ the canonical valuation and
$|\;|_K:K^\times\ra \Q^\times$ the absolute value given by
$|x|_K=\#(k(\nu_K))^{-\nu_K(x)}$. We let $|\;|_K^{1/2}:K^\times \ra
\R^\times_{>0}$ denote the unique positive unramified square root of $|\;|_K$. If $K$ is clear from the context, we
will sometimes write $|\;|$ for $|\;|_K$. We let $\Frob_K$ denote the
geometric Frobenius element of $G_{k(\nu_K)}$ and $I_K$ the kernel of
the natural surjection $G_K\onto G_{k(\nu_K)}$. We will sometimes
abbreviate $\Frob_{\Q_p}$ by $\Frob_p$. We let $W_K$ denote the
preimage of $\Frob_K^\Z$ under the map $G_K\onto G_{k(\nu(K))}$,
endowed with a topology by decreeing that $I_K\subset W_K$ with its
usual topology is an open subgroup of $W_K$. We let $\Art_K : K^\times
\iso W_K^\ab$ denote the local Artin map, normalized to take
uniformizers to lifts of $\Frob_K$.

Let $\Omega$ be an algebraically closed field of characteristic $0$. A
Weil-Deligne representation of $W_K$ over $\Omega$ is a triple
$(V,r,N)$ where $V$ is a finite dimensional vector space over
$\Omega$, $r:W_K\ra GL(V)$ is a representation with open kernel and
$N:V\ra V$ is an endomorphism with
$r(\sigma)Nr(\sigma)^{-1}=|\Art^{-1}_K(\sigma)|_K N$. We say that
$(V,r,N)$ is Frobenius semisimple if $r$ is semisimple. We let
$(V,r,N)^{\Fsemis}$ denote the Frobenius semisimplification
of $(V,r,N)$ (see for instance Section 1 of \cite{ty}) and we let $(V,r,N)^{\semis}$ denote
$(V,r^{\semis},0)$. If $\Omega$ has the
same cardinality as $\C$, we have the notions of a Weil-Deligne
representation being \emph{pure} or \emph{pure of weight $k$} -- see
the paragraph before Lemma 1.4 of \cite{ty}.

We will let $\rec_K$ be the local Langlands correspondence of
\cite{ht}, so that if $\pi$ is an irreducible complex admissible
representation of $GL_n(K)$, then $\rec_K(\pi)$ is a Weil-Deligne
representation of the Weil group $W_K$.  We will write $\rec$ for
$\rec_K$ when the choice of $K$ is clear. If $\rho$ is a continuous
representation of $G_K$ over $\barQQ_l$ with $l\neq p$ then we will
write $\WD(\rho)$ for the corresponding Weil-Deligne representation of
$W_K$. (See for instance Section 1 of \cite{ty}.) 

If $m\geq 1$ is an integer, we let $\Iw_{m,K}\subset GL_m(\cO_K)$
denote the subgroup of matrices which map to an upper triangular
matrix in $GL_m(k(\nu_K))$. If $\pi$ is an irreducible
admissible supercuspidal representation of $GL_m(K)$ and $s\geq 1$ is
an integer we let $\Spp_s(\pi)$ be the square integrable
representation of $GL_{mr}(K)$ defined for instance in Section I.3 of
\cite{ht}. Similarly, if $r : W_K \ra GL_m(\Omega)$ is an irreducible
representation with open kernel and $\pi$ is the supercuspidal
representation $\rec_K^{-1}(r)$, we let $\Sp_s(r)=\rec_K(\Sp_s(\pi))$.
If $K'/K$ is a finite extension and if $\pi$ is an irreducible smooth
representation of $GL_n(K)$ we will write $\BC_{K'/K}(\pi)$ for the
base change of $\pi$ to $K'$ which is characterized by
$\rec_{K'}(\pi_{K'})= \rec_K(\pi)|_{W_{K'}}$.

If $\rho$ is a continuous de Rham representation of $G_K$ over
$\barQQ_p$ then we will write $\WD(\rho)$ for the corresponding
Weil-Deligne representation of $W_K$ (its construction, which is due
to Fontaine, is recalled in Section 1 of \cite{ty}), and if $\tau:K \into \barQQ_p$
is a continuous embedding of fields then we will write
$\HT_\tau(\rho)$ for the multiset of Hodge-Tate numbers of $\rho$ with
respect to $\tau$. Thus $\HT_\tau(\rho)$ is a multiset of $\dim \rho$
integers.  In fact, if $W$ is a de Rham representation of $G_K$ over
$\barQQ_p$ and if $\tau:K \into \barQQ_p$ then the multiset
$\HT_\tau(W)$ contains $i$ with multiplicity $\dim_{\barQQ_p} (W
\otimes_{\tau,K} \widehat{\barK}(i))^{G_K} $. Thus for example
$\HT_\tau(\epsilon_l)=\{ -1\}$.

If $F$ is a number field and $v$ a prime of $F$, we will often denote
$\Frob_{F_v}$, $k(\nu_{F_v})$ and $\Iw_{m,F_v}$ by $\Frob_v$, $k(v)$
and $\Iw_{m,v}$. If $\sigma:F \into \barQQ_p$ or $\C$ is an embedding of fields,
then we will write $F_\sigma$ for the 
closure of the image of $\sigma$. If $F'/F$ is a soluble, finite Galois extension and
if $\pi$ is a cuspidal automorphic representation of $GL_m(\A_F)$ we
will write $\BC_{F'/F}(\pi)$ for its base change to $F'$, an
automorphic representation of $GL_n(\A_{K'})$. If $R:G_F \ra
GL_m(\Qlbar)$ is a continuous representation, we say that $R$ is {\it
  pure of weight $w$} if for all but finitely many primes $v$ of $F$,
$R$ is unramified at $v$ and every eigenvalue of $R(\Frob_v)$ is a
Weil $(\#k(v))^w$-number. (See Section 1 of \cite{ty}.)  If $F$ is an
imaginary CM field, we will denote its maximal totally real subfield
by $F^+$ and let $c$ denote the non-trivial element of $\Gal(F/F^+)$.

\section{Automorphic Galois
  representations}
\label{sec:irreducibility}
\setcounter{subsection}{1}

We recall some now-standard
notation and terminology. 
Let $F$ be an imaginary CM field or totally real field. Let $F^+$ denote the
maximal totally real subfield of $F$. 
By a {\em RAECSDC} or {\em RAESDC} (regular, algebraic, essentially
(conjugate) self dual, cuspidal) automorphic representation of
$GL_m(\A_{F})$ we mean a pair $(\Pi,\chi)$ where
\begin{itemize}
\item $\Pi$ is a cuspidal automorphic representation of $GL_m(\A_F)$ such that  $\Pi_\infty$ has the same infinitesimal character as some irreducible
algebraic representation of the restriction of scalars from $F$ to $\Q$ of
$GL_m$,
\item $\chi:\A_{F^+}^\times/(F^+)^\times \ra \C^\times$ is a
  continuous character such that $\chi_v(-1)$ is independent of $v|\infty$,
\item and $\Pi^c \cong \Pi^\vee \otimes (\chi \circ \norm_{F/F^+} \circ \det)$.
\end{itemize}
If $\chi$ is the trivial character we will often drop it from the
notation and refer to $\Pi$ as a RACSDC or RASDC (regular, algebraic, (conjugate)
self dual, cuspidal) automorphic representation. We will say that
$(\Pi,\chi)$ has {\em level prime to $l$} (resp. {\em level
  potentially prime to $l$}) if for all $v|l$ the representation
$\Pi_v$ is unramified (resp. becomes unramified after a finite base
change).

If $\Omega$ is an algebraically closed field of characteristic $0$ we will
write $(\Z^m)^{\Hom(F,\Omega),+}$ for the set of $a=(a_{\tau,i}) \in (\Z^m)^{\Hom(F,\Omega)}$ satisfying
\[ a_{\tau,1} \geq \dots \geq a_{\tau,m}. \]
Let $w \in \Z$. If $F$ is totally real or imaginary CM (resp. if $\Omega=\C$) we will write $(\Z^m)^{\Hom(F,\Omega)}_w$
for the subset of elements $a \in (\Z^m)^{\Hom(F,\Omega)}$ with 
\[ a_{\tau,i}+a_{\tau \circ c, m+1-i}=w \]
(resp. 
\[ a_{\tau,i}+a_{c \circ \tau, m+1-i}=w.) \]
(These definitions are consistent when $F$ is totally real or
imaginary CM and $\Omega=\C$.)
If $F'/F$ is a finite extension we define $a_{F'} \in (\Z^m)^{\Hom(F',\Omega),+}$ by
\[ (a_{F'})_{\tau,i} = a_{\tau|_F,i}. \] Following \cite{shin} we will
be interested, \emph{inter alia}, in the case that either $m$ is odd;
or that $m$ is even and for some $\tau \in \Hom(F,\Omega)$ and for some
odd integer $i$ we have $a_{\tau,i}>a_{\tau,i+1}$. If either of these
conditions hold then we will say that $a$ is {\em Shin-regular}. (This is often referred to as `slightly regular' in the literature. However as this notion is strictly stronger than `regularity' we prefer the terminology `Shin-regular'.) Following \cite{blggt}, we say that
$a$ is {\em extremely regular} if for some $\tau$ the $a_{\tau,i}$
have the following property: for any subsets $H$ and $H'$ of $\{
a_{\tau,i}+n-i\}_{i=1}^n$ of the same cardinality, if $\sum_{h \in H}
h = \sum_{h \in H'} h$ then $H=H'$.

If $a \in (\Z^m)^{\Hom(F,\C),+}$, let $\Xi_a$ denote the irreducible algebraic representation of
$GL_m^{\Hom(F,\C)}$ which is the tensor product over $\tau$ of the
irreducible representations of $GL_n$ with highest weights $a_\tau$. We will
say that a RAECSDC automorphic representation $\Pi$ of $GL_m(\A_F)$ has
{\em weight $a$} if $\Pi_\infty$ has the same infinitesimal character as
$\Xi_a^\vee$. Note that in this case $a$ must lie in $(\Z^m)^{\Hom(F,\C)}_w$ for some $w \in \Z$.

We recall (see  Theorem 1.2 of \cite{blght}) that to a RAECSDC or RAESDC automorphic representation $(\Pi,\chi)$ of
$GL_m(\A_F)$ and $\imath:\Qlbar \iso \C$ we can associate a continuous semisimple
representation
\[ r_{l,\imath}(\Pi): \Gal(\barF/F) \lra GL_m(\Qlbar) \]
with the properties described in Theorem 1.2 of \cite{blght}. In particular
\[ r_{l,\imath}(\Pi)^c \cong r_{l,\imath}(\Pi)^\vee \otimes \epsilon_l^{1-m}
r_{l,\imath}(\chi)|_{G_F},\] 
where $r_{l,\imath}(\chi): G_{F^+}\to \Qlbar^\times$ is as
defined in Section 1 of \cite{blght}.
For $v|l$ a place of $F$, the representation
$r_{l,\imath}(\Pi)|_{G_{F_v}}$ is de Rham and if $\tau:F \into \barQQ_l$ then
\[ \HT_\tau(r_{l,\imath}(\pi))=\{ a_{\imath\tau,1}+m-1, a_{\imath \tau,2}+m-2,...,a_{\imath \tau,m} \}. \]
If $v\ndiv l$, then the main result of \cite{ana} states that
\[  \imath\WD(r_{l,\imath}(\Pi)|_{G_{F_v}})^{\Fsemis} \cong \rec(\Pi_v \otimes
|\det|^{(1-m)/2}). \] 

Let $p$ be a prime number,  $K/\Q_p$ be a finite extension and let $\Omega$ be an
algebraically closed field of characteristic $0$. Let $\cJ$ denote the
set of equivalence classes of irreducible representations of $W_K$
over $\Omega$ with open kernel, where $s \sim s'$ if
$s\cong s'\otimes \chi\circ \det$ for some unramified
character $\chi:K^\times \to \Omega^\times$.  Let $\rho=(V,r,N)$ be a
Weil-Deligne representation of $W_K$ over $\Omega$. We decompose
\[ V \cong \bigoplus_{\sigma\in\cJ}V[\sigma] \]
where $V[\sigma]$ is the largest $W_K$-submodule of $V$ with all its
irreducible subquotients lying in $\sigma$. Then each $V[\sigma]$ is
stable by $N$ and $\rho[\sigma]:=(V[\sigma],r|_{V[\sigma]},N|_{V[\sigma]})$ is a Weil-Deligne
subrepresentation of $(V,r,N)$. For each
$\sigma\in\cJ$ with $V[\sigma]\neq (0)$, there is a unique decreasing
sequence of integers $m_1(\rho,\sigma)\geq \dots \geq
m_{n(\rho,\sigma)}(\rho,\sigma)\geq 1$ with
\[ \rho[\sigma]^{\Fsemis} \cong \bigoplus_{i=1}^{n(\rho,\sigma)} \Sp_{m_i(\rho,\sigma)}(s_i)\]
$s_i\in \sigma$ for each $i$. If $\rho'$ is another
Weil-Deligne representation of $W_K$ over $\Omega$, we say that
\[ \rho \prec \rho'\]
if $\rho^{\semis}\cong (\rho')^{\semis}$ and if for each
$\sigma\in\cJ$ we have
\[ m_1(\rho,\sigma)+\dots+m_i(\rho,\sigma) \leq m_1(\rho',\sigma)+\dots+m_i(\rho',\sigma)\]
for each $i\geq 1$. The goal of this paper is to establish the following local-global
compatibility result at places dividing $l$.

We now state our main theorem.
\begin{thm}
  \label{main-thm}
Let $(\Pi,\chi)$ be a RAECSDC automorphic representation of
$GL_m(\A_F)$ and let $\imath:\Qlbar\iso\C$. If $v|l$ is a place of
$F$, then
\[ \imath\WD(r_{l,\imath}(\Pi)|_{G_{F_v}})^{\Fsemis}\prec
\rec(\Pi_v\otimes|\det|^{(1-m)/2}).\]
Furthermore, if $\Pi$ has Shin-regular weight, then 
\[ \imath\WD(r_{l,\imath}(\Pi)|_{G_{F_v}})^{\Fsemis}\cong
\rec(\Pi_v\otimes|\det|^{(1-m)/2}).\]
\end{thm}

The following corollary follows immediately using base change as in Proposition 4.3.1 of \cite{cht}.
\begin{cor}
  \label{main-cor}
Let $(\Pi,\chi)$ be a RAESDC automorphic representation of
$GL_m(\A_F)$ and let $\imath:\Qlbar\iso\C$. If $v|l$ is a place of
$F$, then
\[ \imath\WD(r_{l,\imath}(\Pi)|_{G_{F_v}})^{\Fsemis}\prec
\rec(\Pi_v\otimes|\det|^{(1-m)/2}).\]
Furthermore, if $\Pi$ has Shin-regular weight, then 
\[ \imath\WD(r_{l,\imath}(\Pi)|_{G_{F_v}})^{\Fsemis}\cong
\rec(\Pi_v\otimes|\det|^{(1-m)/2}).\]
\end{cor}

\section{The extremely regular, Shin-regular case}\label{sec:reducing to the
  semistable case}
\setcounter{subsection}{1}

We start by treating the special case where, thanks to the irreducibility results of \cite{blggt},  we can give a direct argument. We use an analogue of the trick of \cite{kisindefrings} and \cite{MR2538615} (see the proof of
Theorem 4.3 of \cite{kisindefrings} and Section 2.2 of
\cite{MR2538615}), but in a situation where we need to use a non-abelian, indeed non-Galois, base change. Because of this the argument makes essential use of the irreducibility results of \cite{blggt}, and hence at present can only be made in the extremely regular case.

\begin{thm}
  \label{thm:ext regular, sl regular case}
  Let $m\geq 2$ be an integer, $l$ a rational prime and $\imath:
  \Qlbar \isoto \bb C$.  Let $F$ be an imaginary CM field and $(\Pi,\chi)$ a RAECSDC
  automorphic representation of $GL_m(\A_F)$. If $\Pi$ has extremely
  regular and Shin-regular weight and $v|l$ is a place of $F$, then
\[ \imath\WD(r_{l,\imath}(\Pi)|_{G_{F_v}})^{\Fsemis} \cong \rec(\Pi_{v} \otimes |\det |^{(1-m)/2}). \] 
\end{thm}

\begin{proof}
  We first reduce to the RACSDC case:
  using Lemma 4.1.4 of \cite{cht} we choose an algebraic Hecke
  character $\psi : \A_F^\times/F^\times \to \C^\times$ such that
  $\psi\cdot (\psi\circ c) =\chi_F^{-1}\circ \norm_{F/F^+}$. Then $\Pi
  \otimes \psi \circ \det$ is RACSDC and the theorem holds for $\Pi$
  if and only if it holds for $\Pi \otimes \psi\circ \det$. We may
  therefore assume that $\Pi$ is RACSDC.
  
  To prove the theorem, it suffices to establish the weaker result that
  \[ \imath\WD(r_{l,\imath}(\Pi)|_{G_{F_v}})^{\semis} \cong
  \rec(\Pi_{v} \otimes |\det |^{(1-m)/2})^\semis. \]
  (Suppose this weaker result holds. By Proposition 1.1 of \cite{blggtlocalglobalI}, it suffices to prove that
  $\WD(r_{l,\imath}(\Pi)|_{G_{F_v}})$ is pure. This is established
  in Corollary 1.3 of \cite{blggtlocalglobalI}.)  

  To establish the weaker result, it suffices to show that
  \[ \tr(\sigma | \imath\WD(r_{l,\imath}(\Pi)|_{G_{F_v}})) = \tr (\sigma |
  \rec(\Pi_{v} \otimes |\det |^{(1-m)/2})) \] for every
  $\sigma \in W_{F_v}$ mapping to a non-zero power of $\Frob_v \in
  G_{k(v)}$. (This follows from the proof of Lemma 1 of \cite{saito}.) Fix such an element $\sigma \in W_{F_v}$. We can and do
  choose a finite extension $E_v/F_v$ inside $\barL_v$ such that
  \begin{itemize}
  \item $\sigma \in W_{E_v} \subset W_{F_v}$ and
  \item  $\BC_{E_v/F_v}(\Pi_{v})^{\Iw_{m,E_v}}\neq\{0\}$.
  \end{itemize}
  (If we write $\WD(r_{l,\imath}(\Pi)|_{G_{F_v}})=(V,r,N)$, we could
  take $E_v$ to be the fixed field of the subgroup of $W_{F_v}$
  generated by $\sigma$ and the kernel of $r|_{I_{F_v}}$.) Let
  $E_v'/E_v$ denote the normal closure of $E_v/F_v$. Choose a finite
  CM soluble Galois extension $F'/F$ such that for each place $w|v$ of
  $F'$, $F'_w/F_v \cong E_v'/F_v$. Let $\Pi_{F'} = \BC_{F'/F}(\Pi)$. By
  Theorem 5.4.2 of \cite{blggt} we can and do choose a rational prime
  $l'$ and $\imath':\Q_{l'}\isoto \C$ such that
  $r_{l',\imath'}(\Pi_{F'})$ is irreducible. Choose a prime $w|v$ of $F'$
  and an $F_v$-embedding $F'_w \into \barF_v$. Let $E=F'\cap E_v \subset F'_w$ be the
  fixed field of $\Gal(F'_w/E_v)\subset \Gal(F'/F)$. The inclusion
  $E\into E_v$ determines a prime $u$ of $E$. By Lemma 1.4 of
  \cite{blght}, there exists a RACSDC automorphic representation
  $\Pi_E$ of $GL_m(\A_E)$ with $r_{l',\imath'}(\Pi_E)\cong
  r_{l',\imath'}(\Pi)|_{G_E}$ and hence $r_{l,\imath}(\Pi_E)\cong
  r_{l,\imath}(\Pi)|_{G_E}$. Since $\Pi_{E,u}^{\Iw_{m,u}}\neq \{0\}$,
  Theorem 1.2 of \cite{blggtlocalglobalI} implies that
  \begin{eqnarray*}
    \tr(\sigma| \imath\WD(r_{l,\imath}(\Pi)|_{G_{F_v}})) 
   &=& \tr(\sigma| \imath\WD(r_{l,\imath}(\Pi_E)|_{G_{E_u}})) \\
   &=& \tr(\sigma| \rec(\Pi_{E,u}\otimes|\det|^{(1-m)/2} )) \\
   &=& \tr(\sigma| \rec(\Pi_{v}\otimes|\det|^{(1-m)/2} )) ,
  \end{eqnarray*}
and the result follows.
\end{proof}

\section{The general case}\label{sec:Chenevier-Harris}
\setcounter{subsection}{1} 
We will prove the next result using Theorem \ref{thm:ext regular, sl
  regular case} and the methods of \cite{chenevier} and
\cite{bellchenbk}. It establishes the first statement of Theorem \ref{main-thm}.

\begin{thm}
  \label{thm:general case}
  Let $m\geq 2$ be an integer, $l$ a rational prime and $\imath:
  \Qlbar \iso \bb C$.  Let $F$ be an imaginary CM field and $(\Pi,\chi)$ a RAECSDC
  automorphic representation of $GL_m(\A_F)$. If $v|l$ is a place of $F$, then
\[ \imath\WD(r_{l,\imath}(\Pi)|_{G_{F_v}})^{\Fsemis}\prec
\rec(\Pi_{v} \otimes |\det |^{(1-m)/2}). \]
\end{thm}

Before giving the proof, we first deduce the second statement of
Theorem \ref{main-thm} as a corollary.

\begin{cor}
  \label{cor: main result, Shin-regular case}
Let $m\geq 2$ be an integer, $l$ a rational prime and $\imath:
  \Qlbar \iso \bb C$.  Let $F$ be an imaginary CM field and $(\Pi,\chi)$ a RAECSDC
  automorphic representation of $GL_m(\A)$. If $\Pi$ has slightly
  regular weight and $v|l$ is a place of $F$, then
\[ \imath\WD(r_{l,\imath}(\Pi)|_{G_{F_v}})^{\Fsemis} \cong \rec(\Pi_{F,v} \otimes |\det |^{(1-m)/2}). \] 
\end{cor}

\begin{proof}
  This follows immediately from Theorem \ref{thm:general case} together with Corollary
  1.3 of \cite{blggtlocalglobalI} and Proposition 1.1 of \cite{blggtlocalglobalI}.
\end{proof}

\begin{proof}[Proof of Theorem \ref{thm:general case}]
  As in the proof of Theorem \ref{thm:ext regular, sl regular case},
  we may assume that $\Pi$ is RACSDC. Replacing $F$ by a suitable finite soluble CM Galois extension in which $v$ splits we may also assume that:
  \begin{itemize}
  \item $[F^+:\Q]$ is even;
  \item $F/F^+$ is unramified at all finite places;
 \item all places of $F^+$ dividing $l$ are split in $F$;
  \item if $\Pi_w$ is ramified, then $w|_{F^+}$ is split in $F$;
  \item if $w\neq v$ then $\Pi_{u}^{\Iw_{m,w}}\neq\{0\}$.
  \end{itemize}

Since $[F^+:\Q]$ is even,we can and do
choose a unitary group $U/F^+$ such that:
\begin{itemize}
\item $U\times_{F^+}F \cong GL_m/F$;
\item $U\times_{F^+}F^+_u$ is quasi-split for each prime $u$ of $F^+$;
\item $U(F^+_\sigma)$ is compact for each $\sigma : F^+\into \R$.
\end{itemize}
(We write $F^+_\sigma$ for the completion of $F^+$ with respect to the absolute value induced by $\sigma$.) 
For each place $u$ of $F^+$ which splits in $F$ and $w|u$ a prime of $F$ , we fix
an isomorphism $\imath_w : U(F^+_u)\isoto GL_m(F_w)$ such that
$\imath_{w^c}={}^t\imath_w^{-c}$. If $\sigma : F^+\into \R$ and
$\tsigma:F\into \C$ extends $\sigma$, we fix an embedding
$\imath_{\tsigma}:U(F^+_\sigma)\into GL_n(F_{\tsigma})$ which
identifies $U(F^+_\sigma)$ with the set of all $g$ with ${}^tg^c\cdot g=1_m$.
By Corollaire 5.3 and Th\'eor\`eme 5.4 of \cite{labesse}, there
exists an automorphic representation $\pi_F$ of $U(\A_{F^+})$ such
that:
\begin{itemize}
\item if $u$ is a prime of $F^+$ which splits as $ww^c$ in $F$, then
  $\pi_{F,u}\cong \Pi_{F,w}\circ \imath_w$;
\item if $u$ is a prime of $F^+$ which is inert in $F$, then
  $\Pi_{F,u}$ is given by the local base change of $\pi_{F,u}$ (see \cite{labesse});
\item if $\sigma : F^+\into \R$ and $\tsigma:F\into \C$ extends
  $\sigma$, then there is an irreducible algebraic representation
  $W_{\tsigma}$ of $GL_m(F_{\tsigma})$ such that $\pi_{F,\sigma}\cong
    W_{\tsigma}^{\vee}\circ \imath_{\tsigma}$. Moreover, if
    $W_{\tsigma}$ has highest weight
    $a_{\tsigma}=(a_{\tsigma,1},\ldots,a_{\tsigma,m})$, then $\Pi_F$ has weight $a=(a_{\tsigma})_{\tsigma:F\into\C}$.
\end{itemize}

We now follow the arguments of \cite{chenevier}. We have chosen
to closely follow \cite{chenevier} even when we could somewhat
simplify the argument in the case of interest to us, in order to ease
comparison with that paper. We note however we take the prime $p$ of
\cite{chenevier} to be the prime $l$ of this paper.
Make the
following definitions: let $\tS_l$
(resp.\ $\tS_v$) denote the set of primes of $F$ dividing $l$ but not
lying above $v$ (resp.\ lying above $v$). Let $\tR$ denote the set of
primes $w$ of $F$ not dividing $l$ and with $\Pi_{F,w}$ ramified. Set
$\tS = \tS_v \cup \tR$. Let $S_l$, $S_v$, $R$ and $S$ denote the sets
of primes of $F^+$ lying under $\tS_l$, $\tS_v$, $\tR$ and $\tS$
respectively. For each $u \in S_l\cup S$, fix a prime $\tu$ of $F$
dividing $u$. We will henceforth identify $U(F^+_u)$ and
$GL_m(F_{\tu})$ via $\imath_{\tu}$ for $u \in S_l\cup S$.

Fix embeddings $\imath_\infty : \Qbar \into \C$ and $\imath_l : \Qbar
\into \Qlbar$ such that $\imath\circ\imath_l = \imath_\infty$. For $u|l$ a prime of $F^+$, following
\cite{chenevier}, we let $\Sigma(u)\subset
\Hom(F^+,\Qlbar)$ denote the set of embeddings inducing $u$ and let
$\Sigma_{\infty}(u)=\imath\Sigma(u)\subset
\Hom(F^+,\R)$. Let $W_{\infty}$ denote the representation
$\otimes_{u\in S_v, \sigma \in \Sigma_{\infty}(u)}\pi_{F,\sigma}$ of $\prod_{u\in S_v, \sigma \in \Sigma_{\infty}(u)}
U(F^+_\sigma)$.

Let $K^S=\prod_{u\not\in S}K_u \subset U(\A^{\infty,S}_{F^+})$ be a
compact open subgroup with
\begin{itemize}
\item $K_u = \Iw_{m,\tu}$ if $u \in S_l$;
\item $K_u$ a hyperspecial maximal compact subgroup of $U(F^+_u)$ otherwise.
\end{itemize}
Let $\cH^{S\cup S_l}=\Z[U(\A_{F^+}^{\infty,S\cup S_l})//K^{S\cup S_p}]$ denote the commutative spherical Hecke algebra. For $u \not \in
S\cup S_l$, let $e_u \in \cH(U(F^+_u))$ be the idempotent
corresponding to $K_u$.

Choose a finite Galois extension $E/\Q$ in $\Qbar$ such that $\Pi_{F,\tu}$ can
be defined over $E$ for each $u \in S$. For $u\in S$, let $\cB_u$
denote the subcateogry of the category of smooth
$E$-representations of $GL_m(F_{\tu})$ determined by the supercuspidal
support of $\Pi_{F,\tu}$ (see Proposition-d\`efinition 2.8 of \cite{bern}).  Let $\gz_u$ denote the center of the
category $\cB_u$. For $u \in R$, let $e_u$ denote the idempotent in
$\cH(GL_m(F_{\tu}))$ corresponding to $\Iw_{m,\tu}$. For $u \in S_v$, choose
an idempotent $e_u$ in $\cH(GL_m(F_{\tu}))$ such that
\begin{itemize}
\item $b_ue_u=e_u$ where $b_u\in \cH(GL_m(F_{\tu}))$ is the projector to $\cB_u$;
\item $e_u \Pi_{F,\tu}\neq \{0\}$;
\item for every irreducible $\pi \in \cB_u\otimes_{E,\imath_{\infty}} \C$, if $e_u \pi \neq \{0\}$,
  then $$\rec(\pi)\prec_{I}\rec(\Pi_{F,\tu}).$$
(We refer to Section 3.10 of \cite{chenevier} for the definition of
$\prec_I$ and to Section 3.6 of {\it op.\ cit.\ } for the fact that
one can choose such an idempotent $e_u$.)
\end{itemize}
Extending $E$ if necessary, we may assume that $e_u$ is defined over
$E$ for each $u\in S$ and we set $e = \otimes'_{u \not \in S_l}e_u$ and 
\[ \cH = \cH^{S\cup S_l}\otimes_{\bb Z}  (\bigotimes_{u \in
    S}{}_E  \,\,\, \gz_u ). \]

Let $L_E$ denote the Galois closure (over $\Q_l$) of the closure of $\imath_l(EF)$ in
$\Qlbar$. Let $T$ denote the diagonal maximal torus in $\prod_{u \in S_l}GL_m(F_{\tu})$
and let $\cT = \Hom(T,\G_m^\rig)$ denote the rigid analytic space over
$\Q_l$ parametrizing continuous $l$-adic characters of $T$. 

Let $\cA$ denote the set of automorphic representations $\pi$ of
$U(\A_{F^+})$ for which $e(\pi^{\infty})^{K_{S_l}}\neq \{0\}$ and
$\otimes_{\sigma \in \Sigma_{\infty}(u), u\in S_v}\pi_{\sigma}\cong
W_\infty$. If $\pi \in \cA$, then $\cH$ acts on $e(\pi^{\infty,S_l})$
through an $E$-algebra homomorphism $\psi_{\C}(\pi):\cH \ra \C$ (this
follows from the fact that $\pi_u^{K_u}$ is 1-dimensional for $u\not
\in S\cup S_l$ while $\gz_u$ acts on $\pi_u$ through a character for
$u\in S$). We define $\psi(\pi):\cH\ra \Qlbar$ to be $\imath^{-1}\circ \psi_\C(\pi)$.

If $\pi \in \cA$, we associate to it an algebraic character $\kappa(\pi) \in
\cT(L_E)$ as in Section 1.4 of \cite{chenevier} (this character records
the highest weights of the representations $\pi_{\sigma}$ for $\sigma
\in \Sigma_\infty(u)$ and $u\in S_l$). If $u \in S_l$ and
$\pi_{\tu}$ is an irreducible smooth representation of $GL_m(F_{\tu})$
with $\pi_{\tu}^{\Iw_{m,\tu}}\neq \{0\}$, an \emph{accessible refinement}
of 
$\pi_{\tu}$ is an unramified character
$\chi_{\tu}:T_m(F_{\tu})\ra\C^\times$such that $\pi_{\tu}$ embeds as
a subrepresentation of
$\nind_{B_m(F_{\tu})}^{GL_m(F_{\tu})}\chi_{\tu}$. (Such a character
always exists.) If $\pi \in \cA$,
then an accessible refinement of $\pi$ is a character $\chi = \prod_{u\in
  S_l}\chi_{\tu} : T=\prod_{u\in S_l}T_m(F_{\tu}) \ra \Qlbar^\times$
where each
$\chi_{\tu}:T_m(F_{\tu})\ra\Qlbar^\times$ is unramified and $\imath
\chi_{\tu}$ is an accessible refinement of
$\pi_{\tu}\otimes|\det|^{(1-m)/2}$. Given such a pair $(\pi,\chi)$, 
we associate to it the character 
$\nu(\pi,\chi) \in \cT(\Qlbar)$ as in Section 1.4 of
\cite{chenevier}.

We let 
\[ \cZ \subset \Hom_E(\cH,\Qlbar)\times \cT(\Qlbar) \]
denote the set of all pairs $(\psi(\pi),\nu(\pi,\chi))$ where
$\pi\in\cA$ and $\chi$ is an accessible refinement of $\pi$.

 By Th\'eor\`eme 1.6 of \cite{chenevier}, the data
$(S_l,W_{\infty},\cH,e)$ determines a four-tuple $(X,\psi,\nu,Z)$ where:
\begin{itemize}
\item $X$ is a reduced rigid analytic space over $L_E$ which is
  equidimensional of dimension $n\sum_{u \in S_l}[F^+_u:\Q_l]$;
\item $\psi : \cH \ra \cO(X)$ is a ring homomorphism with
  $\psi(\cH^{S\cup S^l})\subset \cO(X)^{\leq 1}$;
\item $\nu :X\ra \cT$ is a finite analytic morphism;
\item $Z\subset X(\Qlbar)$ is a Zariski-dense accumulation subset of
  $X(\Qlbar)$ such that the map
\[ X(\Qlbar) \ra \Hom_E(\cH, \Qlbar) \times \cT(\Qlbar) \]
which sends $x \mapsto (h\mapsto (\psi(h)(x), \nu(x)))$ induces a bijection $Z
\isoto \cZ$. (We refer to Section 1.5 of \cite{chenevier} for the definition of
`Zariski-dense accumulation'.) We henceforth identify $Z$ and $\cZ$.
\end{itemize}

If $\pi\in \cA$, then by Corollaire 5.3 of \cite{labesse} there exists
a partition $m=m_1+\ldots+m_r$ of $m$ and conjugate self-dual discrete automorphic
representations $\Pi_i$ of $GL_{m_i}(\A_F)$ such that
$\Pi=\Pi_1\boxplus\dots\boxplus\Pi_r$ is a strong base change of
$\pi$. Let $\Sigma = \tS\cup\tS_l$ and let $F_\Sigma$ denote the
maximal extension of $F$ which is unramified outside $\Sigma$.
Let $G_{F,\Sigma}=\Gal(F_\Sigma/F)$. By Theorem 3.2.5 of \cite{chenevierharris} and the argument of
Theorem 2.3 of \cite{guerberoff2009modularity}, there is a continuous semisimple
representation $r_{l,\imath}(\pi):G_{F,\Sigma} \ra GL_m(\Qlbar)$ with 
\[ \imath\WD(r_{l,\imath}(\pi)|_{G_{F_w}})^{\semis}\cong
\rec(\Pi_w\otimes |\det|^{(1-m)/2})^{\semis}\]
for each prime $w\nmid l$ of $F$. Moreover,
there is a unique continuous $m$-dimensional pseudo-representation
$T:G_{F,\Sigma}\ra \cO(X)$ such that $T_z =
\tr(r_{l,\imath}(\pi))$ for each $z=(\psi(\pi),\nu(\pi,\chi)) \in
Z$. (Here, for any $x\in X(\Qlbar)$, $T_x$ denotes the composition of $T$ with
the evaluation map $\cO(X)\ra\Qlbar; g\mapsto g(x)$.) The existence of
$T$ follows from the proof of Proposition 7.1.1 of \cite{chenevier2}
together with Proposition
7.2.11 of \cite{bellchenbk} (which shows that $\cO(X)^{\leq 1}$ is
compact, as $\cT$ is nested and $\nu$ is finite) and the fact that $\psi(\cH^{S\cup
  S^l})\subset \cO(X)^{\leq 1}$. By Theorem 1 of \cite{taylor-siegel}, for any $x\in X(\Qlbar)$, there is
a unique continuous semisimple representation $r_x : G_{F,\Sigma}\ra GL_m(\Qlbar)$ with $T_x=\tr(r_x)$.

Now, let $u\in S$. By Proposition 3.11 of \cite{chenevier}, there is a unique
$m$-dimensional pseudo-character 
\[ T^{\cB_u} : W_{F_{\tu}}\ra \gz_u \]
such that for each irreducible smooth representation $\pi_{\tu}$ of
$GL_m(F_{\tu})$ in $\cB_u\otimes_{E,\imath_{\infty}} \C$, if $T^{\cB_u}_{\pi_{\tu}}$ denotes the
composition of $T^{\cB_u}$ with the character $\gz_u\ra\C$  giving the action of $\gz_u$ on $\pi_{\tu}$, then
\[ T^{\cB_u}_{\pi_{\tu}} = \tr(\rec(\pi_{\tu}\otimes|\det|^{(1-m)/2})).\]

Let $z_0 \in Z$ be a point corresponding to $\pi_F$ together with the
choice of some accessible refinement. Let $Z^\reg\subset Z$ denote the
subset associated to pairs $(\pi,\chi)$ where $\pi_\infty$ is slightly
regular and extremely regular. (If $\tsigma : F \into \C$ and $\sigma:
=\tsigma|_{F^+}$, then $\pi_{\sigma}\circ \imath_{\tsigma}$ is the
restriction of an irreducible algebraic representation of
$GL_m(F_{\tsigma})$ of highest weight $b_{\tsigma}$, say. We say
$\pi_{\infty}$ is Shin-regular or extremely regular if
$b:=(b_{\tsigma})_{\tsigma}$ has the corresponding property.)  Then
$Z^\reg$ is a Zariski-dense accumulation subset of $X(\Qlbar)$. Choose
an open affinoid $\Omega\subset X$ such that $z_0\in \Omega$ and
$Z^{\reg}\cap \Omega$ is Zariski-dense in $\Omega$. Let $T_{\Omega}$
denote the restriction of $T$ to $\Omega$. By Lemme 7.8.11 of
\cite{bellchenbk}, there exists a reduced, separated, quasi-compact
rigid analytic space $Y$ and a proper, generically finite, surjective morphism $f: Y\ra
\Omega$ such that there exists an $\cO_Y$-module $M$ which is locally
free of rank $n$ and carries a continuous action of $G_{F,\Sigma}$
whose trace is given by $f^* T_{\Omega}$.

By Proposition 3.16 of \cite{chenevier} (a result of Sen), the
(generalized) Hodge-Tate weights of $M_y|_{G_{F_{\tu}}}$ are
independent of $y\in Y(\Qlbar)$. (This follows from the quoted result and the fact that the
Hodge-Tate weights of $r_z|_{G_{F_{\tu}}}$ are independent of $z\in
Z$.) Moreover, by the improvement to Theorem C of \cite{bercol} made in
Corollary 3.19 of \cite{chenevier}, there exists a finite Galois extension
$F'_{\tu}/F_{\tu}$ such that if $F'_{\tu,0}\subset F'_{\tu}$ denotes
the maximal unramified extension of $\Q_l$, then the
$\cO_Y\otimes_{\Q_l}F'_{\tu,0}$-module
\[ D_{\st}^{F'_{\tu}}(M):= (M\otimes_{\Q_l}B_{\st})^{G_{F'_{\tu}}} \] 
is locally free of rank $m$ and satisfies the following:
if $y\in Y(\Qlbar)$, then the natural map $D_{\st}^{F'_{\tu}}(M)_y \ra
  D_{\st}^{F'_{\tu}}(M_y)$ is an isomorphism (and hence $M_y|_{G_{F'_{\tu}}}$ is semistable).

The diagonal action of $G_{F_{\tu}}$ on $M\otimes_{\Q_l}B_{\st}$
induces an $\cO_Y$-linear, $F'_{\tu,0}$-semilinear action of
$G_{F_{\tu}}$ on $D^{F'_{\tu}}_\st(M)$. We define an
$\cO_Y\otimes_{\Q_l}F'_{\tu,0}$-linear action $r_{\tu}$ of
$W_{F_{\tu}}\subset G_{F_{\tu}}$ on $D^{F'_{\tu}}_\st(M)$ by letting
$g \in W_{F_{\tu}}$ act as $g\circ \varphi^{w(g)}$ where $w(g)\in\Z$
is the power of $\Frob_l$ to which $g$
maps in $G_{F_{\tu}}/I_{F_{\tu}}$. We have that $N\circ
r_{\tu}(g)=l^{w(g)}r_{\tu}(g)\circ N$ on $D^{F'_{\tu}}_{\st}(M)$. For
each continuous embedding $\tau : F'_{\tu,0}\into L_E$, we let
\[ \WD_{\tu,\tau} =
D_\st^{F'_{\tu}}(M)\otimes_{\cO_Y\otimes_{\Q_p}F'_{\tu},1\otimes
  \tau}\cO_Y.\] Then $\WD_{\tu,\tau}$ is locally free of rank $m$ as
an $\cO_Y$-module and $N\circ r_{\tu}(g)=l^{w(g)}r_{\tu}(g)\circ N$ on
$\WD_{\tu,\tau}$. Moreover, $\varphi$ induces an isomorphism
$\WD_{\tu,\tau\circ\Frob_l}\isoto \WD_{\tu,\tau}$ compatible with
$r_{\tu}$ and $N$. We let $\WD_{\tu}$ denote $\WD_{\tu,\tau}$ for some
choice of $\tau$, regarded as a $W_{F_{\tu}}$-module with an operator
$N$. We note that for each $y\in Y(\Qlbar)$, $\WD_{\tu,y}$ is the
Weil-Deligne representation associated to $M_y|_{G_{F_{\tu}}}$. It
follows that $N^m=0$ on $\WD_{\tu}$.  Let
\[ T^{Y,\tu}=\tr(r_{\tu}(\cdot)|\WD_{\tu}): W_{F_{\tu}} \ra \cO_Y. \]
We claim that
\[ T^{Y,\tu} = f^*\circ \psi \circ T^{\cB_u}. \]
This is proved as follows: let $y \in f^{-1}(Z^\reg\cap\Omega)$ and
let $z=f(y)$. Then $z$ corresponds to a pair $(\pi,\chi)$ where $\pi\in\cA$
is Shin-regular and extremely regular (and $\chi$ is an accessible refinement of
$\pi$). Theorem \ref{thm:ext regular, sl regular case} together with the regularity conditions satisfied by $\pi$ and the
construction of the representation $r_{l,\imath}(\pi)$ in the proof of
Theorem 2.3 of \cite{guerberoff2009modularity} show that
\[ \WD(r_{l,\imath}(\pi)|_{G_{F_{\tu}}})^{\Fsemis}\cong
\imath^{-1}\rec(\pi_u\circ \imath_{\tu}^{-1}\otimes |\det|^{(1-m)/2}).\]
Since $M_y^{\semis}\cong r_z=r_{l,\imath}(\pi)$, we deduce that
$T^{Y,\tu}(g)$ and $f^*(\psi(T^{\cB,u}(g)))$ agree on $y\in Y(\Qlbar)$
for each $g\in W_{F_{\tu}}$. The claimed result now follows from the
Zariski-density of $f^{-1}(Z^\reg\cap\Omega)$ in $Y$.

We now choose some $y_0\in Y(\Qlbar)$ with $f(y_0)=z_0$. Since
$r_{l,\imath}(\Pi_F)=r_{l,\imath}(\pi_F)=r_{z_0}\cong M_{y_0}^{\semis}$, the result just
proved shows that
\[ \imath\WD(r_{l,\imath}(\Pi_F)|_{G_{F_{\tu}}})^{\semis}\cong
\rec(\Pi_{F,\tu}\otimes|\det|^{(1-m)/2}).\]
We deduce from
this that 
\[ \imath\WD(r_{l,\imath}(\Pi_F)|_{G_{F_{\tu}}})^{\Fsemis}\prec
\rec(\Pi_{F,\tu}\otimes|\det|^{(1-m)/2}),\]
as follows: By Lemma 3.14(ii) of \cite{chenevier}, it suffices to
show that 
\[ \imath\WD(r_{l,\imath}(\Pi_F)|_{G_{F_{\tu}}})^{\Fsemis}\prec_I
\rec(\Pi_{F,\tu}\otimes|\det|^{(1-m)/2}).\]
 (See the paragraph of \cite{chenevier} two before Lemma 3.14 for the definition of $\prec_I$.) For each $y \in f^{-1}(
Z^\reg\cap \Omega)$ with $z=f(y)$ corresponding to a pair $(\pi,\chi)$, we have 
\[ \imath\WD(M_y^{\semis}|_{G_{F_{\tu}}})^{\Fsemis}\cong
\rec(\pi_u\circ\imath_{\tu}^{-1}\otimes|\det|^{(1-m)/2})\prec_I
\rec(\Pi_{F,\tu}\otimes|\det|^{(1-m)/2})\]
(where the last relation follows from the choice of idempotent
$e_u$). By the proof of Proposition 7.8.19(iii) of \cite{bellchenbk} and the
Zariski-density of $f^{-1}(Z^{\reg}\cap\Omega)$ in $Y$, we have
$\imath\WD(M_y^{\semis}|_{G_{F_{\tu}}})^{\Fsemis}\prec_I
\rec(\Pi_{F,\tu}\otimes|\det|^{(1-m)/2})$ for all $y\in
Y(\Qlbar)$. Taking $y$ above $z_0$ gives the required result.
\end{proof}

\bibliographystyle{amsalpha}
\bibliography{barnetlambgeegeraghty}

\end{document}